\documentclass[a4paper,10pt]{article}

\usepackage{setspace}
\usepackage[a4paper,left=4cm, right=4cm, top=5.2cm, bottom=5.3cm]{geometry}
\usepackage{amsmath,amssymb}
\usepackage{amsthm}
\usepackage{array,graphicx}
\usepackage{booktabs}
\usepackage{pifont}
\usepackage{tabu}
\usepackage{longtable}
\usepackage{xcolor}
\usepackage{tcolorbox}
\usepackage{textcomp}
\usepackage{caption}
\usepackage{tikz}
\usepackage[normalem]{ulem}

\usepackage[utf8]{inputenc}
\usepackage[english]{babel}

\newtheorem{theorem}{Theorem}
\newtheorem{lemma}{Lemma}

\theoremstyle{definition}

\title{\textbf{\LARGE On Traczyk's BCK-sequences}} 

\author{ Denis Zelent} 

 \date{}

\begin{document}
\maketitle

\noindent
\textbf{Abstract.} {\footnotesize 
BCK-sequences and $n$-commutative BCK-algebras were introduced by T. Traczyk, together with two related problems. The first one, whether BCK-sequences are always prolongable. The second one, if the class of all $n$-commutative BCK-algebras is characterised by one identity. W. A. Dudek proved that the answer to the former question is positive in some special cases, e.g. when BCK-algebra is linearly ordered. T. Traczyk showed that the answer to the latter is affirmative for $n=1,2$. Nonetheless, by providing counterexamples, we proved that the answers to both those open problems are negative.\\ 
}

\bigskip\bigskip\bigskip
Various types of BCK-algebras – as algebras strongly connected with nonclassical
propositional calculi – are studied by many authors. A short survey
of basic results on BCK-algebras can be found in the book \cite{book}. The class of BCK-algebras is not a variety, but, for example, the class of finite BCK-algebras is solvable \cite{solv}. Such BCK-algebras have important applications in coding theory \cite{Flaut2} (see also \cite{Flaut} and \cite{Saeid}). For this reason people are looking for new ways of defining various classes of BCK-algebras to make this study easier, as in e.g. \cite{Evans}, where the method of rooted trees is used to construct commutative BCK-algebras.
In these studies, it is important whether a given class of BCK-algebras can be defined with a small number of simple identities. T. Traczyk showed \cite{Traczyk88} that for $n = 1$ and $n = 2$ the class of all $n$ -commutative BCK algebras can be defined with only one identity. We will show that for $n> 2$ this is, unfortunately, no longer the case.

By a BCK-algebra we mean an algebra of the form $(X,\cdot,0)$, where $X$ is the non-empty set with a designated element $0$ and the dot as a binary operation satisfying the following axioms:
$$
\begin{array}{lllllll}
(1) \ \ ((x\cdot y)\cdot(z \cdot y)) \cdot (x\cdot z) = 0,& \ \ \ \ &(2) \ \   (x\cdot(x\cdot y))\cdot y = 0,\\
(3) \ \  x\cdot 0 = x,&&(4) \ \   0\cdot x = 0,\\
(5) \ \ x\cdot y = y\cdot x = 0 \implies x=y.
\end{array}
$$
Then also $x\cdot x=0$. Moreover, any BCK-algebra $(X,\cdot,0)$ is partially ordered by the relation $\leq$ defined by
 \begin{equation}\nonumber
    x\leq y \text{ iff } x\cdot y = 0.
\end{equation}

 A congruence $\rho$ defined on a partially ordered algebra $(X,\cdot,\leq)$ is convex if and only if for all $x,y,z\in X$ from $x\leq y\leq z$ and $(x,z)\in \rho$ it follows $(x,y)\in\rho$ (cf. \cite{Fuks}). In the case of BCK-algebras, a congruence $\rho$  is convex if and only if $(x\cdot y,0)\in\rho$  together with $(y\cdot x,0)\in\rho$ imply $(x,y)\in\rho$ (cf. \cite{Traczyk85}). H. Yutani proved in \cite{Yutani} that all congruences of a finite BCK-algebra are convex. T. Traczyk obtained (cf. \cite{Traczyk88}) a more general result: a BCK-algebra in which every strongly decreasing (with respect to $\leq$) sequence of elements is finite has only convex congruences. This prompted T. Traczyk to study BCK-algebras in which certain sequences stabilise from a certain point. Such algebras are e.g. {\em $n$-commutative BCK-algebras}, i.e. BCK-algebras, in which $n$ is a minimal integer for which for every two elements $x_0,x_1$ such that $x_1\leq x_0$ we have $x_n =x_{n+1}$, where $x_k =x_{k-2}\cdot(x_{k-2}\cdot x_{k-1})$ for $k=2,3,\ldots$ The class $\mathbf{V}_n$ of all $n$-commutative BCK-algebras is a variety and $\mathbf{V}_n\ne\mathbf{V}_{n+1}$ (cf. \cite{Traczyk88}). Moreover, if for arbitrary $x,y$ in a given BCK-algebra we define two \textsl{BCK-sequences} $x_0, x_1, x_2, \dots$ and $y_o, y_1, y_2, \dots$ by

\medskip
\begin{tabular}{cl}
(6)&    $x_0=x, \ x_1 = y\cdot (y\cdot x),  \dots, x_k = x_{k-2}\cdot (x_{k-2}\cdot x_{k-1}), \dots$\\[1mm]
(7)&$    y_0=y, y_1 = x\cdot (x\cdot y), \dots, y_k = y_{k-2}\cdot (y_{k-2}\cdot y_{k-1}), \dots$
\end{tabular}

\medskip
\noindent
for $k=2,3,\ldots$ 

Then 

\bigskip

\begin{tabular}{lllll}
\begin{tabular}{cl}
(8)&$x_{0}\ge y_{1}\ge x_{2}\ge y_{3}$,\\[1mm]
(9)&$y_{0}\ge x_{1}\ge y_{2}\ge x_{3}$.
\end{tabular}
\\
&\hspace*{20mm}&

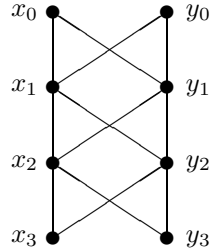
\unitlength=0.5mm
\begin{picture}(80,20)
\put(0,0){\line(0,1){60}} \put(30,0){\line(0,1){60}}
\put(0,0){\circle*{3.5}} \put(0,0){\line(3,2){30}}
\put(30,0){\circle*{3.5}} \put(30,0){\line(-3,2){30}}
\put(-11,-1){$x_3$} \put(35,-1){$y_3$}
\put(0,20){\circle*{3.5}} \put(0,20){\line(3,2){30}}
\put(30,20){\circle*{3.5}} \put(30,20){\line(-3,2){30}}
\put(-11,19){$x_2$} \put(35,19){$y_2$}
\put(0,40){\circle*{3.5}} \put(0,40){\line(3,2){30}}
\put(30,40){\circle*{3.5}} \put(30,40){\line(-3,2){30}}
\put(-11,39){$x_1$} \put(35,39){$y_1$}
\put(0,60){\circle*{3.5}} \put(30,60){\circle*{3.5}}
\put(-11,59){$x_0$} \put(35,59){$y_0$}
\end{picture}

\end{tabular}

\bigskip
The variety $\mathbf{V}_1$ is characterised by the identity $x_1=y_1$; the variety $\mathbf{V}_2$ by the identity $x_2=y_2$ (cf. \cite{Traczyk88}).

Due to this fact, T. Traczyk posed in \cite{Traczyk88} the following two questions:

\medskip
\noindent
{\sc Question 1.} {\em Can the sequences $(8)$ and $(9)$ always be prolonged?}

\medskip
\noindent
{\sc Question 2.} {\em  Is the variety $\mathbf{V}_n$ characterised by the identity $x_n=y_n$?}

\medskip
As for the first question, a partial answer was given by W.A. Dudek. Namely, he proved in \cite{Unsolved} that prolongation of $(8)$ and $(9)$ is possible in BCK-algebras satisfying the identity $x\cdot(x\cdot y)=y\cdot(y\cdot x)$ and in BCK-algebras that are linearly ordered. He also gave an example of a BCK-algebra with infinite strongly decreasing sequences $(8)$ and $(9)$. Nevertheless, the answer to Question 1 is negative. 

\begin{theorem}
For every $n\geq 6$ there are at least two BCK-algebras of order $n$ for which the sequences $(8)$ and $(9)$ cannot be prolonged.
\end{theorem}

\begin{proof}Consider two non-isomorphic BCK-algebras:

\medskip\noindent
\begin{minipage}[c]{0.5\textwidth}
\centering\small
\begin{tabular}{c|cccccc} 
$\cdot$&0&1&2&3&4&5\\\hline
0& 0& 0& 0& 0& 0& 0\\
1& 1& 0& 0& 0& 0& 0\\
2& 2& 1& 0& 0& 0& 0\\
3& 3& 1& 1& 0& 0& 0\\
4& 4& 2& 1& 1& 0& 1\\
5& 5& 3& 2& 1& 1& 0\\
\end{tabular}
\captionof{table}{}
\label{table1}
\end{minipage}
\begin{minipage}[c]{0.5\textwidth}
\centering\small
\begin{tabular}{c|cccccc} 
$\cdot$&0&1&2&3&4&5\\\hline
0& 0& 0& 0& 0& 0& 0\\
1& 1& 0& 0& 0& 0& 0\\
2& 2& 2& 0& 0& 0& 0\\
3& 3& 2& 1& 0& 1& 1\\
4& 4& 4& 4& 4& 0& 0\\
5& 5& 4& 4& 4& 1& 0\\
\end{tabular}
\captionof{table}{}
\label{table2}
\end{minipage}
\\
\newline

They were found as counterexample to Question 1 using computer program written by the author.

The BCK-algebra from Table \ref{table1} has two maximal elements (with respect to $\leq$): $x_0=4$ and  $y_0=5$. For these elements, using $(6)$ and $(7)$, we obtain:
$$
\begin{array}{lllllllll}
x_0=4&x_1=3,&x_2=2,&x_k=1&{\rm for} \ k\geq 3\\[3pt]
y_0=5,&y_1=2,&y_2=2& y_k=2&{\rm for} \ k\geq 3.
\end{array}
$$
Thus $(8)$ and $(9)$ have the form
$$
x_0=4\geq 2\geq 2\geq 2, \ \ \ \ \ y_0=5\geq 3\geq 2\geq 1
$$
and  cannot be prolonged because $y_3\cdot x_4 = 1$, i.e,  $y_3 \nleq x_4$.

The BCK-algebra from Table \ref{table2} also has two maximal elements (with respect to the order $\leq$): $x_0=3$ and $y_0=5$. For these elements we have\\
$$
\begin{array}{rlll}
    x_0=3,& x_k=1 \text{ for } k\geq 1,\\[3pt]
    y_0=5,& y_1=2, \ y_2=1, \ y_k=0 \text{ for } k\geq 3.
\end{array}
$$
Since $x_3\cdot y_4 = 1$, these sequences cannot be prolonged. Thus, for $n = 6$, there are two BCK-algebras with BCK-sequences that cannot be prolonged. 

Now let $(G_n,\cdot,0)$ be an arbitrary BCK-algebra of order $n\geq 6$. Consider the set $G_{n+1}=G_n\cup\{n\}$ and the multiplication
$$
x*y=\left\{\begin{array}{clll}
x\cdot y&{\rm for} \ x,y\in G_n,\\
0&{\rm for} \ x\in G_{n+1}, \ y=n,\\
n&{\rm for} \ x=n, \ y\in G_n.
\end{array}\right.
$$
It is not difficult to verify that $(G_{n+1},*,0)$ is a BCK-algebra of order $n+1$ and $(G_n,\cdot,0)$ is its BCK-subalgebra. 

If $G_6$ is a BCK-algebra defined by Table \ref{table1} (or by Table \ref{table2}), then $G_7$ is a BCK-algebra in which the sequences $(8)$ and $(9)$ initiated by $x_0=4$, $y_0=5$ (respectively, by $x_0=3$, $y_0=5$) cannot be prolonged. By induction, these sequences cannot be prolonged in each BCK-algebra $G_{n+1}$, $n\geq 6$.
\end{proof}

\begin{lemma}\label{L-1}
The set $X_n=\{0,1,2,\ldots,n-1\}$,  $n\geq 5$, with the operation
$$x* y=\left\{\begin{array}{llll}
0&{ for} & x\leq y,\\
x&{for} &y=0,\\
1&{for} & x=y+1,\\
x-y-1&{for} & x-y-1>0
\end{array}\right.
$$
is a BCK-algebra linearly ordered by the natural order of non-negative integers.
\end{lemma}
\begin{proof}
Because axioms $(3)$, $(4)$ and $(5)$ are trivial, we  will check only axioms $(1)$ and $(2)$. For $x=0$ or $y=0$ the condition $(1)$ is valid for each $z\in X_n$. Substituting $z=0$ we can reduce it to $(x* y)*x=0$, which is true for $x\leq y$. If $y>x$, then $(x* y)* x=1* x=0$ for $x=y+1$, and $(x*y)*x=(x-y-1)*x=0$ otherwise. Thus, it is true for $z=0$. It is also true when it contains only two different elements. 

The remaining case is when $x, y, z$ are three different non-zero elements. The cases $x<y<z,$ \ $x<z<y$ and $z<x<y$ are trivial. 

Let $A=((x*y)*(z*y))*(x*z)$. If $z<y<x$, then $y\geq z+1$, $x\geq z+2$. Hence $x*z=x-z-1>0.$  Thus, 
$A=(x*y)*(x-z-1)=0$  for $x=y+1$. For $x>y+1$ we have $A=(x-y-1)*(x-z-1)=0$ since $x-y<x-z$.
So, in this case $(1)$ is satisfied.
 
If $y<x<z$, then $x\geq y+1$, $z\geq y+2$ and $z*y>0$.
Thus $A=1*(y*z)=0$ for $x=y+1$, and $A=(x-y-1)*(z-y-1)=0$ for $x>y+1$ since $x-y<z-y$. Hence, in this case, $(1)$ is satisfied as well.
 
Now let $0<y<z<x$, meaning $x-y-1>0$. 

For $z=y+1$, $A=((x-y-1)*1)*(x*z)=0$  if $x-y-1=1$ or $x-y-1=2$. If $x-y-1=t\geq 3$, then $x-z=t$. Hence, $A=(t*1)*(x*z)=(t-3)*(t-1)=0$. 

For $z=y+k$, $k>1$, we have $x-y-1=k+t-1>0$, $z-y-1=k-1>0$, $A=((k+t-1)*(k-1))*(x*z)=0$, if $t=1$. If $t>1$, then $A=((k+t-1)*(k-1))*(t-1)=(t-1)*(t-1)=0$. So $(1)$ is satisfied for every case.

To prove $(2)$, let us observe that for any $x\leq y$ as well as for $y=0$, the axiom is always satisfied. For $x=y+1$ we have $(x*(x*y))*y=((y+1)*1)*y=0$. For $x=y+k$, $k>1$, we have $((y+k)*(k-1))*y=y*y=0$.

This completes the proof.
\end{proof}

\begin{lemma} \label{BCK-Vn}
Let $(X_n,\cdot,0)$ be as in the previous lemma. For every  $n\geq 5$, the algebra $(X'_n,*,0)$, where $X'_n=X_n\cup\{n\}$ and 

\begin{minipage}[c]{0.7\textwidth}
\centering
$
x\cdot y=\left\{\begin{array}{llll}
x*y&for&x,y\in X_n,\\
n&for&x=n, \ y=0,\\
n-y-1&for&x=n, \ y\in X_n-\{0\},\\
0&for&x\in X'_n-\{n-1\}, \ y=n,\\
1&for&x=n-1, \ y=n.
\end{array}\right.
$
\end{minipage}
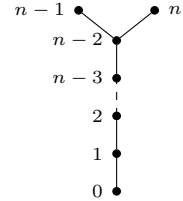
\begin{minipage}[c]{0.3\textwidth}
\centering
       \begin {tikzpicture}
\tikzset{enclosed/.style={draw, circle, inner sep=0pt, minimum size=.1cm, fill=black}}
\node[enclosed, label={left: \scriptsize $0$}] (A) at (0,0) {};
\node[enclosed, label={left: \scriptsize $1$}] (B) at (0,0.5) {};
\node[enclosed, label={left: \scriptsize $2$}] (C) at (0,1) {};
\node[enclosed, label={left: \scriptsize $n-3$}] (G) at (0,1.5) {};
\node[enclosed, label={left: \scriptsize $n-2$}] (D) at (0,2) {};
\node[enclosed, label={left: \scriptsize $n-1$}] (E) at (-0.5,2.4) {};
\node[enclosed, label={right: \scriptsize $n$}] (F) at (0.5,2.4) {};
\path (A) edge [] node[] {} (B);
\path (B) edge [] node[] {} (C);
\path (G) edge [] node[] {} (D);
\path (D) edge [] node[] {} (E);
\path (D) edge [] node[] {} (F);
\draw[dashed] (0,1) -- (0,1.5);
\end{tikzpicture}
\end{minipage}

\noindent 
is a BCK-algebra of order $n+1$.
\end{lemma}

Two examples of such constructed BCK-algebras are shown below:

\medskip\noindent
\begin{minipage}[c]{0.5\textwidth}
\centering
\begin{tabular}{c|cccccc} 
$\cdot$&0&1&2&3&4&5\\\hline
0& 0& 0& 0& 0& 0& 0\\
1& 1& 0& 0& 0& 0& 0\\
2& 2& 1& 0& 0& 0& 0\\
3& 3& 1& 1& 0& 0& 0\\
4& 4& 2& 1& 1& 0& 1\\
5& 5& 3& 2& 1& 1& 0\\
\end{tabular}
\captionof{table}{Case $n=5$}
\end{minipage}
\begin{minipage}[c]{0.5\textwidth}
\centering
\begin{tabular}{c|ccccccc} 
$\cdot$&0&1&2&3&4&5&6\\\hline
0& 0& 0& 0& 0& 0& 0&0\\
1& 1& 0& 0& 0& 0& 0&0\\
2& 2& 1& 0& 0& 0& 0&0\\
3& 3& 1& 1& 0& 0& 0& 0\\
4& 4& 2& 1& 1& 0& 0& 0\\
5& 5& 3& 2& 1& 1& 0& 1\\
6& 6& 4& 3& 2& 1& 1& 0\\
\end{tabular}
\captionof{table}{Case $n=6$}
\end{minipage}

\begin{proof} Due to the way the algebra $(X'_n,\cdot.0)$ is defined, it directly follows that
$$
(10) \ \ x\leq y \implies n*y\leq n*x.
$$
Additionally,
$$
(11) \ \ x\leq n \implies x*y\leq n*y \ {\rm for\; all} \  y\ne n.
$$
Indeed, for $x\leqslant y$ the last implication is trivial. If $y<x$, then $n=x+k$, $x=y+t$, $n=x+k+t$, $k,t>0$, which for $t=1$ gives $x*y=1\leqslant n*y$ since by the definition $n*y\geqslant  1$ for all $y\ne n$. For $t>1$ we have $x*y=x-y-1=t-1<k+t-1=n+y$, which completes the proof of $(11)$.

In view of Lemma \ref{L-1}, the proof that $(X'_n,\cdot,0)$ is a BCK-algebra can be done by verifying $(1)$ and $(2)$, in the case when at least one element is equal to $n$. Conditions $(3)$, $(4)$ and $(5)$ are satisfied due to the method of the above definition. 

If in $(1)$ one element is $n$ and the second is $0$, or one is $n$ and the other two are equal, $(1)$ is satisfied.

Now, let $x=n$. Then $0<y<z<n$ or $0<z<y<n$. The first case needs to be divided into two subcases: 
\begin{itemize}
    \item $z=y+1$. Then $((n*y)*(z*y))*(n*z)=((n*y)*1)*(n*z)=0$
    if $y=n-2$ or $y=n-3$.
        
				If $y<n-3$, then $((n*y)*1)*(n*z)=((n-y-1)*1)*(n*z)=(n-y-3)*(n*z)=(n*(y+2))*(n*z)=0$, where the last equation follows from (10).
    
    \item $z>y+1$. Then $((n*y)*(z*y))*(n*z)=((n*y)*(z-y-1))*(n*z)=
           ((n-y-1)*(n-y-2))*1=1*1=0$ for $z=n-1$.
     
		For $z<n-1$ we have $((n*y)*(z-y-1))*(n*z)=((n-y-1)*(z-y-1))*(n*z)=(n-y-1-(z-y-1)-1)*(n*z)=(n-z-1)*(n*z)=(n*z)*(n*z)=0$.
\end{itemize}

Let $y=n$. Then $0<x<z<n$ or $0<z<x<n$. In the first case $x*n=0$ and thus $((x*n)*(z*n))*(x*z)=0$. For the second case, if $x=n-1$, then $((x*n)*(z*n))*(x*z)=(1*0)*((n-1)*z)=0$ since $(n-1)*z\neq 0$.

Finally, let $z=n$. Then if $0<x<y<n$, then $((x*y)*(n*y))*(x*n)=0$ because $x*y=0$, and if $0<y<x<n$, then $((x*y)*(n*y))*(x*n)=0$ follows from (11).

This completes the proof of $(1)$. As for $(2)$, the cases when $y=n$ or when $x=n$ and $y\in\{0,n-1,n\}$ are trivial. 
The only remaining case is when $x=n$ and $y\in \{1,2,\dots, n-2\}$, but then $(x\cdot(x\cdot y))\cdot y= (n\cdot(n-(y+1))\cdot y=(n-(n-(y+1)+1))\cdot y=y\cdot y = 0$.

    Thus, $(X'_n,\cdot,0)$ is a BCK-algebra.     
\end{proof}

We can now show that the above construction allows us to give a counterexample to Question 2.

\begin{theorem}
For $m\geq 3$, the variety $V_m$ is not determined by $x_m=y_m$.
\end{theorem}
\begin{proof}
We will prove it by showing that for every $n\geq 5$ the BCK-algebra of order $n+1$ defined in Lemma \ref{BCK-Vn} belongs to the variety $\mathbf{V}_{n-2}$, but there exists $x,y$ such that $x_{n-2} \neq y_{n-2}$.

Firstly, we will show that this BCK-algebra belongs to $\mathbf{V}_{n-2}$. From Lemma \ref{BCK-Vn}, $X_n$ and $X'_n-\{n-1\}$ are isomorphic linearly ordered BCK-algebras and thus the longest possible sequence (of different elements) which we can obtain occurs when $x_0=n-1$ and $x_1=n-2$. In that case $x_2=n-3$, $x_3=n-4, \dots, x_{n-3}=2$, $x_{n-2}=1=x_{n-1}$. In any other case, we will also have $x_{n-2}=x_{n-1}$ due to the linearity and the length of those sequences. That shows that this BCK-algebra indeed belongs to $\mathbf{V}_{n-2}$.

Now, let us see what happens with sequences (6) and (7) in case $x=n-1$, $y=n$. Then $x_1 = y\cdot(y\cdot x) = n\cdot (n\cdot (n-1))=n\cdot 1 = n-2$, $x_2=x\cdot(x\cdot x_1)$ $ = (n-1)\cdot ((n-1)\cdot (n-2)) = (n-1)\cdot 1 = n-3, \dots, x_{n-3}=2, \;x_{n-2}=1$, but $y_1= x\cdot(x\cdot y)= (n-1)\cdot ((n-1)\cdot n)= (n-1)\cdot 1 = n-3, \; y_2 = y\cdot(y\cdot y_1)=n\cdot (n\cdot(n-3))=n\cdot 2 = n-3, \dots, y_{n-2} = n-3$, and obviously $n-3\neq 1$ for $n\geq 5$, meaning $x_{n-2} \neq y_{n-2}$ for those sequences, which completes the proof.
\end{proof}

\section*{Conclusion}
This paper shows that although prolonging BCK-sequences is possible in some special cases, as shown in \cite{Unsolved}, it is not possible in general. It also shows that the variety $V_n$ is not generated by the identity $x_n=y_n$. This solves both open problems posed by Traczyk in \cite{Traczyk88}.

\section*{Acknowledgement}
I would like to extend my deepest gratitude to Wiesław A. Dudek for bringing those open problems to my attention, as well as for his guidance in writing this paper.
Special thanks to Michael Kinyon, who was the first person to give the counterexample to Question 2 in the case of $\mathbf{V}_3$.



{\tiny DEPARTMENT OF MATHEMATICAL SCIENCES, NORWEGIAN UNIVERSITY OF SCIENCE AND TECHNOLOGY, NO7491 TRONDHEIM, NORWAY

\textsl{E-mail address: zelden99@zohomail.eu}}
\end{document}